\documentclass[graybox]{svmult}

\usepackage{type1cm}        
\usepackage{makeidx}         
\usepackage{graphicx}        
\usepackage{multicol}        
\usepackage[bottom]{footmisc}

\usepackage{newtxtext}       %
\usepackage[varvw]{newtxmath}       

\makeindex             

\usepackage{amsmath}
\usepackage{amsfonts}

\newtheorem{thm}{Theorem}[section]
\newtheorem{cor}[thm]{Corollary}
\newtheorem{lem}[thm]{Lemma}

\newcommand{\FSSZ} {{\mathbb Z}}

\begin{document}

\bibliographystyle{plain}

\title* 
 {On  Sylvester Equations in Banach Subalgebras}

\author{Qiquan Fang, Chang Eon Shin and Qiyu Sun}

\institute{Qiquan Fang \at Department of Mathematics, Zhejiang University of Science and Technology, Hangzhou, Zhejiang 310023, China,
\email{qiquanfang@163.com} \and Chang Eon Shin \at Department of Mathematics, Sogang University,  Seoul 04107, South Korea, \email{shinc@sogang.ac.kr}
\and
Qiyu Sun\at Department of Mathematics, University of Central Florida, Orlando, FL 32816, USA,
\email{qiyu.sun@ucf.edu}
}






\maketitle

\abstract
{ Let  ${\mathcal B}$ be a Banach algebra and ${\mathcal A}$ be a  Banach subalgebra that admits  norm-controlled inversion in ${\mathcal B}$.
In this work, we take $A, B$ in the Banach subalgebra ${\mathcal A}$
with their spectra in the Banach algebra ${\mathcal B}$ being disjoint, and show that the  operator
Sylvester equation
$ BX-XA=Q$ has a  unique solution $X\in {\mathcal A}$ for every $Q\in {\mathcal A}$.
Under the additional assumptions that ${\mathcal B}$ is the operator algebra ${\mathcal B}(H)$ on a  Hilbert space $H$ and that
  $A$ and $B$ are normal in ${\mathcal B}(H)$,  an explicit norm estimate for  the solution  $X$ of the above operator Sylvester equation is provided in this work.
In addition,   the above conclusion  on norm control is applied to Banach subalgebras of localized infinite matrices and integral operators.
}

\section{Introduction}

Let $m,n\ge 1$ and take matrices $A \in {\mathbb{R}}^{n\times n}$ and $B \in {\mathbb{R}}^{m \times m}$. It is well known that the Sylvester equation
\begin{equation}\label{Sylvester}
BX-XA=Q
\end{equation}
has a unique solution $X\in {\mathbb R}^{m\times n}$  for every $Q \in {\mathbb R}^{m\times n}$ if and only if  $A$ and $B$ have no common eigenvalues \cite{
 Lan, 
 Trent}.
 For the case that both $A={\rm diag}(\lambda_1, \ldots, \lambda_n)$ and $B={\rm diag}(\mu_1, \ldots, \mu_m)$ are diagonal matrices with no common eigenvalues,
one may verify that the solution of the above matrix  Sylvester equation \eqref{Sylvester} is given by
$$ x_{ij}=\frac{q_{ij}} {\mu_i-\lambda_j}\ \ {\rm for\  all} \ 1\le i\le m\ {\rm and} \ 1\le j\le n, $$
where $X=(x_{ij})_{1\le i\le m, 1\le j\le n}$ and $Q=(q_{ij})_{1\le i\le m, 1\le j\le n}$, see \cite[Theorem 3.1]{Sorensen2002} for the
explicit  expression of the solution $X$  
 when $A$ and $B$ are diagonalizable.

The Sylvester equation \eqref{Sylvester} appears in   block diagonalization of  matrices.
Roth's theorem states that  matrices
$\left(\begin{array}{cc} B & Q\\ 0 & A\end{array}\right)$ and
$\left(\begin{array}{cc} B & 0\\ 0 & A\end{array}\right)$
are similar if and only if the Sylvester equation
\eqref{Sylvester} has a solution \cite{roth1952}.  In particular, we have
$$
\left(\begin{array}{cc} B & Q\\ 0 & A\end{array}\right)  =  \left(\begin{array}{cc} I & X\\ 0 & I\end{array}\right)^{-1}\left(\begin{array}{cc} B & 0\\ 0 & A\end{array}\right)  \left(\begin{array}{cc} I & X\\ 0 & I\end{array}\right),$$
where $I$ is the identity matrix of appropriate size.

 The  Sylvester equation \eqref{Sylvester} with assuming $m=n$ and replacing $B$ by $-A^T$  becomes the Lyapunov equation
\begin{equation}\label{Lyapunov}
A^TX+XA+Q=0.
\end{equation}
For the Lyapunov equation \eqref{Lyapunov}, there exists a unique positive definite solution $X$
when  $-A$ and $Q$ are symmetric and positive definite (hence $-A^T$ and $A$ have no common eigenvalues) \cite{barnett70, bunce85, lyapunov, MS2017}.
 In that scenario,
the linear dynamical system
\begin{equation}\label{dynamic.def}
\dot z(t) =Az(t)\end{equation}
is  stable in the sense that $\lim_{t\to +\infty} z(t)=0$.
Define
 the  Lyapunov  function  by
  $$V(z(t))=z(t)^T X z(t),$$
  where $X$ is the  positive definite solution  of the Lyapunov equation \eqref{Lyapunov} with $Q$ replaced by the identity $I$.
 Then it follows from   \eqref{Lyapunov}, \eqref{dynamic.def} and
the positive-definiteness of the matrix $X$ that
 the quadratic function  $V(z)$ satisfies
 $$
 \dot V(z(t))= z(t)^T (A^TX+XA)z(t)=-z(t)^T z(t)\le -\tau V(z(t)),
 $$
where $\tau>0$ is an absolute constant.
Solving the above   differential inequality and using the positive-definiteness of the matrix $X$, we can find a positive constant $C$ such that
\begin{equation}
0\le  z(t)^T z(t)\le C V(z(t)) \le C e^{-\tau t} \ \  {\rm for \ all}\ \ t\ge 0.
\end{equation}
This  proves the 
(exponential) stability of the dynamical system \eqref{dynamic.def}.

\smallskip

Let ${\mathcal B}$ be a Banach algebra. Given $A$ and $B\in {\mathcal B}$, we define
the  Sylvester operator $T_{A, B}$ on ${\mathcal B}$ by
\begin{equation}\label{operator}
T_{A, B}(X)=BX-XA,\quad X \in {\mathcal B}.
\end{equation}
  As the family of all matrices
 $A \in {\mathbb R}^{n\times n}$  forms a Banach algebra, the
Sylvester equation
\eqref{Sylvester}  has been extended in the Banach algebra setting.
 Sylvester-Rosenblum theorem states that the operator Sylvester equation
\begin{equation}\label{sylvester.eq2}
T_{A, B}(X)=Q
\end{equation}
has a unique solution in ${\mathcal B}$ for every $Q \in {\mathcal B}$ (hence
the homogenous Sylvester equation $BX=XA$ has zero as its unique solution) if the spectra $\sigma_{\mathcal B} (A)$ and
$\sigma_{\mathcal B} (B)$
of $A$ and $B$ in ${\mathcal B}$ are disjoint  \cite{Bhatia1997, Rosenblum, sasane2021}.

Let $H$ be  a complex Hilbert space  and  denote   the $C^\star$-algebra of all bounded linear operators on $H$ by
${\mathcal B}(H)$.  We say that  a linear  subspace  $M$ of the Hilbert space $H$  is {\em invariant} under the operator $A\in {\mathcal B}(H)$
if $Ax\in M$ for every $x\in M$, and   {\em hyperinvariant} 
if it is invariant under every operator  $B\in {\mathcal B}(H)$ which commutes with $A$.
One of  the most famous problems in functional analysis is whether every operator on an infinite-dimensional
Hilbert space have a non-trivial invariant subspace.
The Sylvester-Rosenblum theorem provides a sufficient condition on $A\in {\mathcal B}(H)$ so that its invariant space  $M$
is also hyperinvariant. Given $A\in {\mathcal B}(H)$, its invariant subspace $M\subset H$ can be described by
\begin{equation} \label{projection.eq} (I-P) AP=0,\end{equation}
where  $P$ is  the projection operator from $H$  onto $M$.
For any $B\in {\mathcal B}(H)$ commutating with $A$, we have
\begin{align} \label{commutative.eq0}
& ((I-P)A(I-P)+\mu_1P) (I-P)B P \nonumber\\
 = &\  (I-P) A (I-P) BP  = (I-P) ABP\nonumber\\
 = &\  (I-P)B AP  =(I-P)B PAP \nonumber\\
 = &\  (I-P)BP (PAP+\mu_2 (I-P)),
\end{align}
 where $\mu_1, \mu_2\in {\mathbb C}$  and the second and fourth equalities follow from \eqref{projection.eq}.
 By \eqref{commutative.eq0} with $\mu_1$ and $\mu_2$ appropriately chosen and the  Sylvester-Rosenblum theorem for  homogenous operator Sylvester equations, we conclude that
$(I-P)BP=0$ and hence $M$ is hyperinvariant, provided that
the spectra of restrictions of the operator $A$ onto the invariant subspace $M$ and its orthogonal complement  $M^\perp$ are disjoint  \cite{Bhatia1997, radjavl1971}.

\medskip
 The Sylvester-Rosenblum theorem
on the operator Sylvester equation \eqref{sylvester.eq2} also  plays a crucial role to establish
spectral theorem for normal operators on a Hilbert space, see \cite[Section 6]{Bhatia1997}.
The  Sylvester equation appears  in many mathematical fields and engineering applications, including linear algebra, functional analysis,  ordinary differential equation, control theory, and signal processing. For historical remarks and recent advances on Sylvester equations, the readers may refer to
\cite{bertram2021, Bhatia1997, djordjevic2021, 
sasane2021} and references therein.

\smallskip

In this work, we  solve the operator Sylvester equation
\eqref{sylvester.eq2} in an inverse-closed Banach subalgebra and consider norm control of its unique solution.

\section{Main Results}

 Given a Banach algebra ${\mathcal B}$,  we say that its subalgebra ${\mathcal A}$  sharing the same identity $I$ with ${\mathcal B}$  is  {\em inverse-closed} in ${\mathcal B}$  if
 any element in ${\mathcal A}$ that is invertible in ${\mathcal B}$ is also invertible in ${\mathcal A}$.
 Inverse-closedness
 has numerous applications in time–frequency analysis, sampling theory, numerical analysis and optimization. It has been established for infinite matrices, integral operators, and pseudo-differential operators satisfying various off-diagonal decay conditions, see the survey papers
 \cite{grochenig10, Krishtal11, shinsun13} for historical remarks and  \cite{fang21, grochenig2024, MS2017, samei19, shinsun17} and references therein for recent advances.
For an inverse-closed subalgebra ${\mathcal A}$ of a Banach algebra ${\mathcal B}$, it is known that  for any  element $A\in {\mathcal A}$, its spectra in  the algebras ${\mathcal A}$ and
${\mathcal B}$ are the same.  Therefore  by the  Sylvester-Rosenblum theorem, 
 we have the following result on solving the operator Sylvester equation \eqref{sylvester.eq2}  in an inverse-closed subalgebra ${\mathcal A}$.

 \begin{thm}\label{maintheorem0} Let ${\mathcal B}$ be a Banach algebra and  ${\mathcal A}$ be its inverse-closed Banach  subalgebra of ${\mathcal B}$.
If $A, B, Q \in {\mathcal A}$ and the spectra of $A$ and $B$ in ${\mathcal B}$ are disjoint, then
 there is a unique solution  $X$ to the operator Sylvester equation \eqref{sylvester.eq2} in ${\mathcal A}$.
\end{thm}

 A quantitative version of inverse-closedness is  norm-controlled inversion \cite{belinskii97, grochenigklotz13, nikolski99, stafney67}.
Here we say that
an inverse-closed subalgebra ${\mathcal A}$ of a Banach algebra ${\mathcal B}$ admits a {\it norm-controlled inversion}
 if there exists a  nonnegative function
 $h:{\mathbb R}_+ \times {\mathbb R}_+\to {\mathbb R}_+$  that is bounded on any compact subset of ${\mathbb R}_+ \times {\mathbb R}_+$ such that
\begin{equation}\label{controlfunction.def}
\|A^{-1}\|_{\mathcal A}\le h(\|A^{-1}\|_{\mathcal B}, \|A\|_{\mathcal A})
\end{equation}
for all $A\in{\mathcal A}$ that are invertible in  ${\mathcal B}$, where ${\mathbb R}_+$ is the set of all nonnegative real numbers, and $\|\cdot\|_{\mathcal A}$ and $\|\cdot\|_{\mathcal B}$ are norms on Banach algebras ${\mathcal A}$ and
${\mathcal B}$ respectively.
We remark 
that not every inverse-closed Banach subalgebra has norm-controlled inversion. In particular,
 the classical Wiener algebra of
periodic functions with summable Fourier coefficients does not admit a norm-controlled inversion
in the Banach algebra of all  bounded periodic functions.
We also notice that a good estimate for the norm-control function $h$ in \eqref{controlfunction.def} to the inversion is important for some mathematical and engineering applications \cite{cheng2018, sun2014}, and
 the norm-control function associated with some norm-controlled inversion subalgebras may have polynomial growth \cite{fang20,  grochenigklotz10, grochenigklotz14,   shinsun17, shincjfa09}.

Given a complex Hilbert space $H$,
we say that a linear operator $A\in {\mathcal B}(H)$ on $H$ is normal if $A^\star A=A A^\star$ \cite{sheth}.  For a normal operator $A\in {\mathcal B}(H)$, we have
\begin{equation}\label{normal.norm}
\|A\|_{{\mathcal B}(H)}=\sup \{|z|, \ z\in \sigma_{{\mathcal B}(H)}(A)\}.
\end{equation}
 In this work, we consider
solving  the operator Sylvester equation
\eqref{sylvester.eq2} in an inverse-closed subalgebra ${\mathcal A}$
that admits a norm-controlled inversion in  the $C^\star$-algebra ${\mathcal B}(H)$.  

\begin{thm}\label{maintheorem} Let $H$ be a complex Hilbert space,   and  ${\mathcal A}$ be a Banach subalgebra of
the operator algebra ${\mathcal B}(H)$ that admits  norm-controlled inversion in ${\mathcal B}(H)$.
If $A, B, Q \in {\mathcal A}$, and  $A, B$ are normal operators in ${\mathcal B}(H)$ with their spectra  in ${\mathcal B}(H)$ being disjoint, then
 there is a bivariate  function $g$ on ${\mathbb R}_+\times {\mathbb R}_+$ that is bounded on any compact subset of ${\mathbb R}_+ \times {\mathbb R}_+$
   such that
 \begin{equation}\label{maintheorem.eq1}
 \|X\|_{\mathcal A}\le g( (d(A, B))^{-1}, \|A\|_{\mathcal A}+\|B\|_{\mathcal A})
 \end{equation}
 holds for the unique solution $X$ of the operator Sylvester equation \eqref{sylvester.eq2},
 where $d(A, B)$ is the distance of the spectra of $A$ and $B$ in  ${\mathcal B}(H)$.
\end{thm}

Taking $A=0$ (resp.,\ $B=0$), the corresponding operator Sylvester equation \eqref{sylvester.eq2}
becomes a trivial inverse problem $BX=Q$  (resp.,\  $-XA=Q$), and it has a unique solution $X$ in the subalgebra ${\mathcal A}$.
Therefore
 the estimate in \eqref{maintheorem.eq1} for the solution of the operator Sylvester equation \eqref{sylvester.eq2}
 could be considered as the correspondence of the norm estimate  \eqref{controlfunction.def} for the inversion
 in  the  Sylvester setting.  We remark that an appropriate norm estimate for the solution of Sylvester equations
could be essential for some mathematical and engineering applications, such as stability of dynamical systems and optimal control.

The bivariate function $h$ in \eqref{controlfunction.def} is known as a norm-control function of the norm-controlled inversion subalgebra ${\mathcal A}$.
We remark that 
 the norm-control function $h$ can be  chosen so that it is monotonic about every variable, i.e.,
\begin{equation} \label{controlfunction.def2}
0\le h(s_1, t_1)\le h(s_2, t_2) \ \ {\rm if} \ 0\le s_1\le s_2\ {\rm and} \ 0\le t_1\le t_2.
\end{equation}
 Otherwise, we may replace the original norm-control function $h$ by
the following bivariate  function
$$\tilde h(s, t)=\sup_{0\le u\le s, 0\le v\le t} h(u, v) \ {\rm for} \ \ s, t\ge 0,$$
which is well-defined by the boundedness assumption for the original norm-control function on any bounded set.
Applying a similar argument,  we conclude that the function $g$ in \eqref{maintheorem.eq1} could be selected to be  monotonic about every variable.

\smallskip

Let $m\ge 2$ be an integer.
Given a Banach algebra ${\mathcal B}$,  we say that it is a symmetric ${}^\star$-algebra if the spectrum
 $\sigma_{\mathcal B}(A^\star A)$ of $A^\star A$ is contained in $[0,\infty)$ for any
$A\in {\mathcal B}$, and that
its Banach subalgebra  ${\mathcal A}$ is {\it  differential}  if there exist
$\theta \in (0, m-1]$ and an absolute constant $D$ satisfying
\begin{equation}\label{weakpower}
\|A^m\|_{\mathcal A} \le D \|A\|_{\mathcal A}^{m-\theta} \|A\|_{\mathcal B}^{\theta} \ \ \ {\rm for\ all} \ \  A\in {\mathcal A}
\end{equation}
\cite{BC1991, CH1988, grochenigklotz13, KS1994, rief2010, shinsun17}.
In \cite[Theorem 4.1]{shinsun20}, it is shown that a differential ${}^\star$-subalgebra admits a norm-controlled inversion with the norm-control function having subexponential growth.

Combining \cite[Theorem 4.1]{shinsun20} and Theorem \ref{maintheorem}, we have the following corollary on solving the operator Sylvester equation \eqref{sylvester.eq2} in differential ${}^\star$-subalgebra.

\begin{cor}
\label{differential.cor}  Let $H$ be a complex Hilbert space, and ${\mathcal A}$ be a ${}^\star$-subalgebra of ${\mathcal B}(H)$ with common identity and involution ${}^\star$.
If  ${\mathcal A}$ is a differential subalgebra of ${\mathcal B}(H)$, and
 $A, B \in {\mathcal A}$ are normal operators in ${\mathcal B}(H)$ with their  spectra in  ${\mathcal B}(H)$ being disjoint, then
the operator Sylvester equation \eqref{sylvester.eq2} has a norm-controlled solution in the differential ${}^\star$-subalgebra ${\mathcal A}$.
\end{cor}

\section{Sylvester equations for infinite matrices and integral operators}

In this section, we apply the conclusion in Corollary \ref{differential.cor} to solve Sylvester equations in Banach algebras of localized infinite matrices and integral operators.

Let $\ell^p:=\ell^p({\mathbb Z}^d), 1\le p\le \infty$, be the Banach space of all $p$-summable sequences on ${\mathbb Z}^d$ with the norm denoted by $\|\cdot\|_p$.
Given $1\le p\le \infty$ and $\alpha\ge 0$,
 we define
the Gr\"ochenig-Schur  algebra of infinite matrices
by
 \begin{equation}\label{GS.def}
{\mathcal A}_{p,\alpha}=\Big\{ A=(a(i,j))_{i,j \in \FSSZ^d}, \  \|A\|_{{\mathcal A}_{p,\alpha}}<\infty\Big\},
\end{equation}
the Baskakov-Gohberg-Sj\"ostrand algebra of infinite matrices by
\begin{equation}\label{BGS.def}
{\mathcal C}_{p,\alpha}=\Big\{ A= (a(i,j))_{i,j \in \FSSZ^d}, \  \|A\|_{{\mathcal C}_{p,\alpha}}<\infty\Big\},
\end{equation}
and
the Beurling algebra of infinite matrices by
\begin{equation}\label{Beurling.def}
{\mathcal B}_{p,\alpha}=\Big\{ A= (a(i,j))_{i,j \in \FSSZ^d}, \  \|B\|_{{\mathcal B}_{p,\alpha}}<\infty\Big\}
\end{equation}
respectively,
where     $u_\alpha(i, j)=(1+|i-j|)^\alpha, \alpha\ge 0$, are  polynomial weights on $\FSSZ^{d}\times {\mathbb Z}^d$,
 \begin{equation}\label{GSnorm.def}
 \|A\|_{{\mathcal A}_{p,\alpha}}
 = \max \Big\{ \sup_{i \in \FSSZ^d} \big\|\big(a(i,j) u_\alpha(i, j)\big)_{j\in \FSSZ^d}\big\|_p, \ \ \sup _{j \in \FSSZ^d}
 \big\|\big(a(i,j) u_\alpha(i, j)\big)_{i\in \FSSZ^d}\big\|_p
 \Big\},
\end{equation}
\begin{equation}\label{BKSnorm.def}
\|A\|_{{\mathcal C}_{p,\alpha}} = \Big\| \Big(\sup_{i-j=k} |a(i,j)| u_\alpha(i, j)\Big)_{k\in \FSSZ^d} \Big\|_p,
\end{equation}
and
\begin{equation}\label{Beurlingnorm.def}
\|A\|_{{\mathcal B}_{p,\alpha}} = \Big\| \Big(\sup_{|i-j|\ge |k| } |a(i,j)| u_\alpha(i, j)\Big)_{k\in \FSSZ^d} \Big\|_p
\end{equation}
\cite{akramjfa09, baskakov90, Beurling, fang20,  fang21, gohberg89,  gltams06, MS2017,  shinsun17, sjo941,  sunca11, suntams07, sun2005}.
Clearly, we have
\begin{equation}\label{properinclusion}
{\mathcal B}_{p,\alpha} \subset {\mathcal C}_{p,\alpha} \subset
{\mathcal A}_{p,\alpha}  \ \ {\rm for \ all}\ 1\le p\le \infty \ {\rm and} \ \alpha\ge 0.
\end{equation}
The above inclusion become an equality for $p=\infty$, which is
also  known as the Jaffard algebra
\cite{jaffard90}.

For $1\le p\le \infty$ and $\alpha>d-d/p$, it is known that
 ${\mathcal A}_{p,\alpha}$,
  ${\mathcal C}_{p,\alpha}$ and  ${\mathcal B}_{p, \alpha}$  are differential *-subalgebras of ${\mathcal B}(\ell^2)$, the algebra of all bounded linear operators on
   $\ell^2$.
This together with Corollary \ref{differential.cor} yields the following conclusion on solving the operator
Sylvester equation \eqref{sylvester.eq2} in the above three algebras of infinite matrices.

\begin{thm}\label{existence-matrices}
Let $d\ge 1, 1\le p\le \infty, \alpha>d-d/p$,
and ${\mathcal A}$ be either the Gr\"ochenig-Schur  algebra
 ${\mathcal A}_{p,\alpha}$, or the Baskakov-Gohberg-Sj\"ostrand algebra
  ${\mathcal C}_{p,\alpha}$,  or the  Beurling algebra
 ${\mathcal B}_{p, \alpha}$.
If   $A, B \in {\mathcal A}$
have their spectra $\sigma_{{\mathcal B}(\ell^2)}(A)$ and $\sigma_{{\mathcal B}(\ell^2)}(B)$ in ${\mathcal B}(\ell^2)$ being
  disjoint, then for every $Q\in {\mathcal A}$, the operator Sylvester equation \eqref{sylvester.eq2} has a unique solution in ${\mathcal A}$.
    Furthermore, if
  $A, B$ are normal in ${\mathcal B}(\ell^2)$,
  then there is a bivariate function $g$ on ${\mathbb R}_+\times {\mathbb R}_+$ such that
  \eqref{maintheorem.eq1} holds.
\end{thm}

Let ${\mathbb Z}_+^d$ be the set of all $d$-tuples of nonnegative integers, and $L^p:=L^p({\mathbb R}^d),$ $ 1\le p\le \infty$, be
  the space of all $p$-integrable functions on ${\mathbb R}^d$ with its norm denoted by $\|\cdot\|_p$.
Take $1\le p\le \infty$, $\alpha>0$ and  a positive integer $m\ge 1$, and consider Banach algebra
${\mathcal W}_{p, \alpha}^m$ of  localized integral operators
\begin{equation}\label{integral-operator}
Tf(x)=\int_{{\mathbb R}^d} K(x,y)f(y) dy 
\end{equation}
on the space $L^2$ with  the norm defined by
\begin{equation}\label{integral-norm1}
\|T\|_{{\mathcal W}^m_{p,\alpha}}:=
 \max_{k,l\in{\mathbb Z}_+^d, |k|+|l|\le m-1, 0<\delta\le 1}
\|\partial_x^k \partial_y^l K (x, y)\|_{p, \alpha}
+ \delta^{-1}
\|\omega_\delta(\partial_x^k\partial_y^l K(x,y))
\|_{p, \alpha},
\end{equation}
where   $u_\alpha(x, y)=(1+|x-y|)^\alpha$ is a polynomial weight on ${\mathbb R}^{d}\times {\mathbb R}^d$,
 and
for a kernel function $K (x, y), x, y\in {\mathbb R}^{d}$,
we define its modulus of continuity $\omega_\delta(K)$   by
\begin{equation}\label{def-modulus}
\omega_\delta(K)(x,y):=\sup_{|x'|\le \delta, \ |y'|\le \delta}
|K(x+x',y+y')-K(x,y)|,\ x, y\in {\mathbb R}^{d},
\end{equation}
and set
$$ \|K\|_{p, \alpha}:=
 \max\Big(\sup_{x\in {\mathbb R}^d} \big\|K(x,\cdot)u_\alpha(x,\cdot)\big\|_p, \
 \sup_{y\in {\mathbb R}^d}  \big\|K(\cdot,y)u_\alpha(\cdot,y)\big\|_p\Big).
$$
For $\alpha>d-d/p$, ${\mathcal W}_{p, \alpha}^m$  is a non-unital Banach algebra  \cite{fang23, sunacha08}.
Define the unital Banach algebra
${\mathcal {IW}}_{p,\alpha}^m$
 induced from ${\mathcal W}^m_{p,\alpha}$
by
\begin{equation}\label{def-untial-algebra1}
{\mathcal {IW}}_{p, \alpha}^m :=\big\{
\lambda I +T: \ \lambda\in{\mathbb C} \ {\rm and }\
T\in {\mathcal W}^m_{p,\alpha}\big\}
\end{equation}
with
\begin{equation}\label{def-unital1-norm}
\|\lambda I +T \|_{{\mathcal {IW}}_{p,\alpha}^m}:=
|\lambda|+ C_0 \|T\|_{{\mathcal W}^m_{p,\alpha}}
\end{equation}
for some positive constant $C_0$.
With appropriate selection of the constant $C_0$ in \eqref{def-unital1-norm}, one may verify that
 ${\mathcal IW}_{p, \alpha}^m, 1\le p\le \infty, m\ge 1, \alpha>d-d/p$, are
differential ${}^\star$-subalgebra of ${\mathcal B}(L^2)$ (the ${}^\star$-algebra of bounded linear operators on $L^2$) \cite{fang23, sunacha08}.
As a consequence of Corollary  \ref{differential.cor}, we have the following result on the
Sylvester equation \eqref{sylvester.eq2} in  the above algebra of localized integral operators.

\begin{thm}
Let  $1\le p\le \infty, m\ge 1, \alpha>d-d/p$, and let  ${\mathcal IW}_{p, \alpha}^m$ be as in \eqref{def-untial-algebra1}.
If $A, B \in {\mathcal {IW}}_{p,\alpha}^m$ have their spectra in
${\mathcal B}(L^2)$ being disjoint, then there exists a unique solution to the  Sylvester equation \eqref{sylvester.eq2}   for every $Q\in {\mathcal IW}_{p, \alpha}^m$.
    Furthermore, if
  $A, B$ are normal in ${\mathcal B}(L^2)$, then there is a bivariate function $g$ on ${\mathbb R}_+\times {\mathbb R}_+$ such that
  \eqref{maintheorem.eq1} holds.
\end{thm}

The reader may refer to \cite{fang23, fang14, grochenig2024} for additional Banach algebras of localized integral operators and pseudo-differential operators.

\section{Proof of Theorem \ref{maintheorem}}

We say that a bounded and open set $D$ in the complex plane is a  {\it Cauchy domain} if
it contains only a finite number of components with the closures of any two of them being disjoint, and
its boundary $\partial D$
 is composed of a finite positive number of closed positive oriented rectifiable Jordan curves with no two of those curves intersecting.
To prove the main theorem, we need a technical lemma in \cite{Bhatia1997, Rosenblum} for  the unique solution of
 the operator Sylvester equation \eqref{sylvester.eq2}.

\begin{lem}\label{rosenblum.lem}
Let ${\mathcal B}$ be a Banach algebra and $A, B, Q\in {\mathcal B}$. If the spectra  $\sigma_{\mathcal B}(A)$ and $\sigma_{\mathcal B}(B)$
of $A$ and $B$ in the algebra ${\mathcal B}$ are disjoint, then the Sylvester operator $T_{A, B}$ is invertible. Furthermore, for a Cauchy domain $D$ such that
 $\sigma_{\mathcal B}(A) \subset D$,
$\sigma_{\mathcal B}(B)\subset {\mathbb C}\backslash (D\cup \partial D)$,  and its oriented boundary  $\partial D$ has total winding numbers $m\ge 1$ around $\sigma_{\mathcal B}(A)$ and
$ 0$ around $\sigma_{\mathcal B}(B)$, we have
\begin{equation*}
T_{A, B}^{-1}(Q) = - \frac{1}{2m\pi i} \int_{\partial D}(B-zI)^{-1} Q (zI-A)^{-1} dz.
\end{equation*}
\end{lem}

Now we are ready to start the detailed proof of Theorem \ref{maintheorem}.

\begin{proof}[Proof of Theorem \ref{maintheorem}] Set ${\mathcal B}={\mathcal B}(H)$ and define
$$ \delta(A, B)=\min \big\{\max(|\Re z-\Re w|, |\Im z-\Im w|):\ z \in \sigma_{\mathcal B}(A), \ w \in \sigma_{\mathcal B}(B)\big\}$$
Then $\delta(A,B)>0$ by the disjoint assumption for  the spectra of $A$ and $B$ in the Banach algebra ${\mathcal B}$.
Set $\delta'(A, B)= \delta(A, B)/3>0$, let $N_0$ be the integer part of $(\|A\|_{{\mathcal B}}+\delta'(A, B))/\delta'(A, B)$,
and for every $k, k'\in {\mathbb Z}$
 denote the closed square in the complex plane ${\mathbb C}$ with center  $(k+k'i)\delta'(A, B)$ and side length $\delta'(A, B)$
by $S_{k, k'}$. Then we have
\begin{equation}\label{maintheorem.pf.eq2}
\sigma_{\mathcal B}(A)\subset \cup_{-N_0\le k, k'\le N_0} \, S_{k, k'}.
\end{equation}

Define the union of squares $S_{k, k'}, k, k'\in {\mathbb Z}$, having nonempty intersection with $\sigma_{\mathcal B}(A)$ by
$$D_1 =\cup_{S_{k, k'}\cap \sigma_{\mathcal B}(A)\ne \emptyset} \, S_{k, k'}.$$
Similarly, we let
$D_2$ be the union of squares $S_{k, k'}, k, k'\in {\mathbb Z}$, having nonempty intersection with the domain $D_1$,
and  $D_3$ be the union of squares $S_{k, k'}, k, k'\in {\mathbb Z}$ having nonempty intersection with the domain $D_2$.
By the above construction of domains $D_1, D_2, D_3$ and the  distance definition $\delta(A, B)$ between spectra of $A$ and $B$, we obtain
\begin{equation} \label{maintheorem.pf.eq3} \sigma_{\mathcal B}(A)\subset D_1\subset D_2\subset Q(0, (2N_0+3) \delta'(A, B))
\ \ {\rm and} \ \  \sigma_{\mathcal B}(B)\subset {\mathcal C}\backslash D_3\subset {\mathcal C}\backslash D_2,
\end{equation}
where $Q(0, r)$ is the square in the complex plane with center zero and size length $r>0$.

Let $D$ be the interior of the domain $D_2$ with its boundary denoted by $\partial D$.  With the positive direction on the boundary $\partial D$ selected, we see that
 $\partial D$ has total winding numbers $m\ge 1$ around $\sigma_{\mathcal B}(A)$ and
$0$ around $\sigma_{\mathcal B}(B)$.
Furthermore,  $D$ is a Cauchy domain with the boundary $\partial D$ being made of finitely many line segments and the length $\ell(\partial D)$ of its boundary $\partial D$  is
bounded, i.e.,
 \begin{equation}  \label{maintheorem.pf.eq4}
 |z|\le \|A\|_{\mathcal B}+2\sqrt{2} \delta'(A, B)\le \|A\|_{\mathcal B}+ \delta(A, B) \ {\rm for \ all} \ z\in \partial D,
 \end{equation}
 and
\begin{equation}  \label{maintheorem.pf.eq5}
\ell(\partial D)\le 4 (2N_0+3)^2 \delta'(A, B) 
\le 48
 (\|A\|_{{\mathcal B}}+ \delta(A, B))^2 (\delta(A, B))^{-1}.
 \end{equation}
By \eqref{maintheorem.pf.eq3},  the spectra of
$zI-A$ and $zI-B, z\in \partial D$, lie outside the square $Q(0, 2\delta'(A, B))$ in the complex plane  with  the origin as its center and  $2\delta'(A, B)$ as its side length.
This together with the normal property for operators $A$ and $B$ in ${\mathcal B}$ implies that
\begin{align}  \label{maintheorem.pf.eq6}
&  \max(\|(zI-A)^{-1}\|_{{\mathcal B}}, \|(zI-B)^{-1}\|_{{\mathcal B}})\nonumber\\
& \qquad 
\le \max_{w\in \partial Q(0, 2\delta'(A, B))} |w|^{-1}\le  (\delta'(A, B))^{-1} \ {\rm for \ all} \ z\in \partial D.
\end{align}

Let   $h(s, t), s, t\ge 0$,  be the  norm-control function for the subalgebra ${\mathcal A}$   
which satisfies \eqref{controlfunction.def2}.
By  \eqref{maintheorem.pf.eq3} and Lemma \ref{rosenblum.lem},  the unique solution of the operator
Sylvester equation \eqref{sylvester.eq2} in ${\mathcal A}$ is given by
\begin{equation*}
T_{A, B}^{-1}(Q) = -\frac{1}{2m\pi i} \int_{\partial D}(B-zI)^{-1} Q (zI-A)^{-1} dz, \ Q\in {\mathcal B},
\end{equation*}
where $m$ is the total winding numbers $m\ge 1$ of the boundary $\partial D$ around $\sigma_{\mathcal B}(A)$.

Combining the above expression for the solution of the operator Sylvester equation \eqref{sylvester.eq2}   and the estimates in   \eqref{maintheorem.pf.eq4},  \eqref{maintheorem.pf.eq5} and \eqref{maintheorem.pf.eq6}, we obtain
\begin{align*}
\label{maintheorem.pf.eq7}   
\|T_{A, B}^{-1}(Q)\|_{\mathcal A}
& \le (2m\pi)^{-1} \int_{\partial D}\|(B-zI)^{-1}\|_{\mathcal A}
\|Q\|_{\mathcal A} \| (zI-A)^{-1}\|_{\mathcal A} |dz|\nonumber\\
& \le 
24 \pi^{-1} \|Q\|_{\mathcal A}  (\|A\|_{{\mathcal B}}+ \delta(A, B))^2 (\delta(A, B))^{-1} \times \, \nonumber  \\
& 
\quad \Big( h\big( 3 (\delta(A, B))^{-1},\
\|A\|_{\mathcal A}+\|B\|_{\mathcal A}+ (\|A\|_{\mathcal B}+ \delta(A, B))\|I\|_{\mathcal A}\big)\Big)^2.
\end{align*}
This completes the proof.
\end{proof}

\ethics{Competing Interests}{
This work is  partially supported by Zhejiang Provincial Natural Science Foundation of China [LY24A010016] and National Natural Science Foundation of China [12271483, 11701513].
}


\begin{thebibliography}{999}



\bibitem{akramjfa09}
A. Aldroubi, A. Baskakov and I. Krishtal, Slanted matrices, Banach frames,
and sampling, {\em J. Funct. Anal.}, {\bf 255}(2008), 1667--1691.

\bibitem{barnett70} S. Barnett and C. Storey, {\em Matrix Methods in Stability Theory}, Nelson, London, 1970.


\bibitem{baskakov90} A.G. Baskakov, Wiener's theorem and asymptotic estimates for elements of inverse
matrices, {\em Funktsional. Anal. i Prilozhen}, {\bf 24}(1990), 64--65.

 \bibitem{belinskii97}
 E.S. Belinskii, E.R.  Liflyand and R.M.   Trigub,  The Banach algebra
 and its properties, {\em  J. Fourier Anal. Appl.}, {\bf 3}(1997),   103--129.

 \bibitem{bertram2021}  C. Bertram and H. Fassbender,
A quadrature framework for solving Lyapunov and Sylvester equations,
{\em Linear Algebra Appl.}, {\bf   622}(2021),   66--103.


 \bibitem{Beurling} A. Beurling, On the spectral synthesis of bounded functions, {\em Acta Math.},
{\bf 81}(1949), 225--238.

\bibitem{Bhatia1997} R. Bhatia and P. Rosenthal, How and why to solve the operator equation $AX-XB=Y$,
{\em Bull. Lond. Math. Soc.}, {\bf 29}(1997), 1--21.





 \bibitem{BC1991} B. Blackadar and J. Cuntz, Differential Banach algebra norms and smooth subalgebras of $C^\star$-algebras,
{\em J. Operator Theory}, {\bf 26}(1991), 255--282.

\bibitem{bunce85} J.W. Bunce, Stabilizability of linear systems defined over $C^\star$-algebras, {\em  Math. Syst. Theory},
{\bf 18}(1985),  237--250.  



\bibitem{cheng2018} C. Cheng, Y. Jiang and Q. Sun, Spatially distributed sampling and reconstruction, {\em Appl. Comput. Harmon. Anal.}, {\bf 47}(2019),
109--148.

\bibitem{CH1988} M. Christ, Inversion in some algebra of singular integral operators, {\em Rev. Mat. Iberoamericana}, {\bf 4}(1988),
219--225.

\bibitem{djordjevic2021} B.D. Djordjevic,
Singular Sylvester equation in Banach spaces and its applications: Fredholm theory approach,  {\em Linear Algebra Appl.}, {\bf  622}(2021),  189--214.


\bibitem{fang23} Q. Fang, Y. Shen, C.E. Shin and X. Tao,
Norm-controlled inversion in Banach algebras of integral operators,
{\em Banach J. Math. Anal.}, {\bf 17}(2023),  Article No. 21, 29 pp.


\bibitem{fang20} Q. Fang and C.E. Shin,
Norm-controlled inversion of Banach algebras of infinite matrices, {\em Comptes Rendus. Math.}, {\bf 358}(2020),  407--414.


\bibitem{fang21} Q. Fang, C.E. Shin and Q. Sun,
Polynomial control on weighted stability bounds and inversion norms of localized matrices on
simple graph, {\em J. Fourier Anal. Appl.}, {\bf 27} (2021), Article No. 83, 33 pp.

\bibitem{fang14} Q. Fang, C.E. Shin and Q. Sun,  Wiener's lemma for singular integral operators of Bessel potential type,
{\em Monatsch. Math.}, {\bf 173} (2014), 35--54.

\bibitem{gohberg89}  I. Gohberg, M.A. Kaashoek and H.J. Woerdeman, The band method for positive
and strictly contractive extension problems: an alternative version and new
applications, {\em Integral Equations Operator Theory}, {\bf 12}(1989), 343--382.



\bibitem{grochenig10}
K. Gr\"ochenig, Wiener's lemma: theme and variations, an introduction to
spectral invariance and its applications, In {\em Four Short Courses on Harmonic Analysis: Wavelets, Frames, Time-Frequency Methods, and Applications to Signal and Image Analysis}, edited by P. Massopust and B. Forster,
Birkhauser, Boston 2010.


\bibitem{grochenigklotz10} K. Gr\"ochenig and A. Klotz,
Noncommutative approximation: inverse-closed subalgebras and off-diagonal decay of matrices,
{\em Constr. Approx.}, {\bf 32}(2010), 429--466.



\bibitem{grochenigklotz13} K. Gr\"ochenig and A. Klotz, Norm-controlled inversion in smooth Banach algebra I, {\em J. London Math.
Soc.}, {\bf 88}(2013), 49--64.

\bibitem{grochenigklotz14} K. Gr\"ochenig and A. Klotz, Norm-controlled inversion in smooth Banach algebra II, {\em Math. Nachr.},
{\bf 287}(2014), 917--937.



\bibitem{gltams06} K. Gr\"ochenig and M. Leinert, Symmetry of matrix
algebras and symbolic calculus for infinite matrices, {\em Trans.
Amer. Math. Soc.},  {\bf 358}(2006),  2695--2711.

\bibitem{grochenig2024}  K. Gr\"ochenig, C.  Pfeuffer and  J. Toft,  Spectral invariance of quasi-Banach algebras of matrices and pseudodifferential operators,
{\em Forum Math.},  2024, https://doi.org/10.1515/forum-2023-0212



\bibitem{jaffard90} S. Jaffard, Properi\'et\'es des matrices bien
localis\'ees pr\'es de leur diagonale et quelques applications,
{\em Ann. Inst. Henri Poincar\'e}, {\bf 7}(1990), 461--476.




\bibitem{KS1994} E. Kissin and V.S. Shulman, Differential properties of some dense subalgebras of $C^\star$-algebras,
{\em Proc. Edinburgh Math. Soc.}, {\bf37}(1994), 399--422.



\bibitem{Krishtal11} I. Krishtal, Wiener's lemma: pictures at exhibition,
{\em Rev. Un. Mat. Argentina}, {\bf 52}(2011), 61--79.

\bibitem{Lan}
P. Lancaster and  M. Tismenetsky, {\em The Theory of Matrices}, 
2nd ed., Academic Press, 1985. 

\bibitem{lyapunov} A. Lyapunov, Problemes general de la stabilite du mouveement, 1892; reprinted as Ann. of Math.
Stud. 17, Princeton University Press, 1947.

\bibitem{MS2017} N. Motee and Q. Sun, Sparsity and spatial localization measures for spatially
distributed systems, {\em SIAM J. Control Optim.}, {\bf 55}(2017), 200--235.


\bibitem{nikolski99} N. Nikolski, In search of the invisible spectrum, Ann. Inst. Fourier (Grenoble), {\bf 49}(1999), 1925--1998.

\bibitem{radjavl1971} H. Radjavi and P. Rosenthal, Hyperinvariant subspaces for spectral and $n$-normal operators,
{\em Acta Sci. Math. (Szeged)},  {\bf 32}(1971) 121--126.

\bibitem{rief2010} M.A. Rieffel, Leibniz seminorms for ``matrix algebras converge to the sphere'', In {\em Quanta of Maths,
volume 11 of Clay Math. Proc.}, Amer. Math. Soc., Providence, RI,  pp. 543--578, 2010.


\bibitem{Rosenblum}
M. Rosenblum, On the operator eqaution $BX-XA=Q$, {\em Duke Math. J.},
{\bf 23}(1956), 263--269.

\bibitem{roth1952}
W.E. Roth, The equations $AX-YB=C$ and $AX-XB=C$ in matrices, {\em Proc. Amer. Math. Soc.},
{\bf 3}(1952), 392--396.

 \bibitem{samei19}
E. Samei and V. Shepelska, Norm-controlled inversion in weighted convolution algebra, {\em J. Fourier Anal. Appl.}, {\bf 25}(2019), 3018--3044.

\bibitem{sasane2021} A. Sasane,  The Sylvester equation in Banach algebras,
{\em Linear Algebra Appl.}, {\bf  631}(2021),  1--9.



\bibitem{sheth}
I.H. Sheth, On normaloid operators, {\em Pacific J. Math.}, {\bf 28}(1969), 675--676.


\bibitem{shinsun20} C.E. Shin and Q. Sun,
Differential subalgebras and norm-controlled inversion, In {\em Operator Theory, Operator Algebras and Their Interactions with Geometry and Topology}, R. E. Curto, W. Helton, H. Lin, X. Tang, R. Yang and G. Yu eds., Birkhauser Basel, pp. 467--485, 2020.


 \bibitem{shinsun17} C.E. Shin and Q. Sun,
Polynomial control on stability, inversion and powers of matrices on simple graphs,
	 {\em J. Funct. Anal.}, {\bf 276}(2019), 148--182.

\bibitem{shinsun13} C.E. Shin and Q. Sun, Wieners lemma: localization and various approaches,
{\em Appl. Math. J. Chinese Univ.}, {\bf 28}(2013), 465--484.

\bibitem{shincjfa09} C.E. Shin and Q. Sun,
 Stability of localized operators, {\em J. Funct. Anal.}, {\bf 256}(2009), 2417--2439.

\bibitem{sjo941} J. Sj\"{o}strand, Wiener type algebra of pseudodifferential operators, Centre
de Mathematiques, Ecole Polytechnique, Palaiseau France, Seminaire 1994--1995, December 1994.

\bibitem{Sorensen2002} D.C. Sorensen  and  A.C. Antoulas,
The Sylvester equation and approximate
balanced reduction, {\em 	Linear Algebra Appl.},
{\bf  351–352}(2002),  671--700.



\bibitem{stafney67} J.D. Stafney
An unbounded inverse property in the algebra of absolutely convergent Fourier series,
{\em Proc. Amer. Math. Soc.}, {\bf 18}(1967),  497--498.

\bibitem{sun2014} Q. Sun, Localized nonlinear functional equations and two sampling problems in signal processing,
{\em Adv. Comput. Math.},  {\bf 40}(2014), 415--458.

\bibitem{sunca11}
Q. Sun, Wiener's lemma for infinite matrices II, {\em Constr. Approx.}, {\bf 34}(2011), 209--235.

\bibitem{sunacha08} Q.  Sun, Wiener's lemma for localized integral operators, {\em Appl. Comput. Harmonic Anal.}, {\bf 25}(2008), 148--167.


\bibitem{suntams07} Q. Sun, Wiener's lemma for infinite matrices,
{\em Trans. Amer. Math. Soc.},  {\bf 359}(2007), 3099--3123.

\bibitem{sun2005} Q. Sun, Wiener's lemma for infinite matrices with polynomial off-diagonal
decay, {\em C. Acad. Sci. Paris Ser I}, {\bf 340}(2005), 567--570.



\bibitem{Trent}
H.L. Trentelman, A.A. Stoorvogel and M. Hautus,
{\em Control Theory for Linear Systems},
Springer, 
 2001.








\end{thebibliography}
\end{document}